\documentclass{amsart}

\usepackage{amsmath,amssymb,amsfonts,graphics,color}
\usepackage{epsfig}
\usepackage{latexsym}
\usepackage{euscript}
\usepackage{lipsum}
\usepackage{xpatch}
\newcommand{\ignore}[1]{}

\title{Derivations on the algebra of Rajchman measures}

\author[Mahya Ghandehari]{Mahya Ghandehari}
\address{Department of Mathematical Sciences , University of Delaware}

\thanks{At the time this article was written, the author was supported by \emph{University of Delaware Research Foundation} grant number 17A00999.}

\email{mahya@udel.edu}

\date{\today}

\newtheorem{theorem}{Theorem}[section]
\newtheorem{proposition}[theorem]{Proposition}
\newtheorem{lemma}[theorem]{Lemma}
\newtheorem{corollary}[theorem]{Corollary}
\newtheorem{definition}[theorem]{Definition}

\newcommand{\cal}{\mathcal}

\xpatchcmd{\paragraph}{\normalfont}{{\normalfont\bfseries}}{}{}

\def\sup{{\rm sup}}

\def\supp {{\rm supp}}
\def\Re {{\rm Re}}

\def\SL2 {{\rm SL}_2({\mathbb R})}

\begin{document}
\maketitle

\begin{abstract}
For a locally compact Abelian group $G$, the algebra of Rajchman measures, denoted by $M_0(G)$, is the set of all  bounded regular Borel measures on $G$ whose Fourier transform vanish at infinity. 
In this paper,  we investigate the spectral structure of the algebra of Rajchman measures, and illustrate aspects of the residual analytic structure of its maximal ideal space.
In particular, we show that $M_0(G)$ has a nonzero continuous point derivation, whenever $G$ is a non-discrete locally compact Abelian group. 
We then give the definition of the Rajchman algebra for a general (not necessarily Abelian) locally compact group,  and prove that for a non-compact connected SIN group, the Rajchman algebra admits a nonzero continuous point derivation.
Moreover, we discuss the analytic behavior of the spectrum of $M_0(G)$. Namely, we show that for every non-discrete metrizable locally compact Abelian group $G$, the maximal ideal space of $M_0(G)$ contains  analytic discs.
\end{abstract}

\noindent {{\sc AMS Subject Classification:}  Primary 	;}
\newline
{{\sc Keywords:} Rajchman measure, Rajchman algebra, point derivation, analytic disc, strongly independent set.

\section{Introduction \label{sec:intro}}
Let $G$ be a locally compact group. The measure algebra of $G$, denoted by $M(G)$, is  the Banach $*$-algebra of bounded complex Radon measures on $G$. A measure $\mu$ in the measure algebra of a locally compact  Abelian group is called a \emph{Rajchman measure} if its Fourier-Stieltjes transform vanishes at infinity. The importance of Rajchman measures first became apparent in the study of uniqueness of trigonometric series. A subset $E$ of ${\mathbb T}$ is a \emph{set of uniqueness} (or a ${\cal U}$-set) if the trivial series is the only trigonometric series which converges to 0 on every element outside $E$. Sets of uniqueness are typically small. In fact, every Borel ${\cal U}$-set has Lebesgue measure zero. However the converse is not true. In 1916, Menshov  showed that there are closed sets of Lebesgue measure zero which are not sets of uniqueness~\cite{Menshov}. In his proof, Menshov constructed a Rajchman measure $\mu$ which does not belong to $L^1({\mathbb T})$. 

The collection of all Rajchman measures on a locally compact Abelian group $G$ forms a Banach subalgebra of  $M(G)$, which we denote by $M_0(G)$. 
The algebra of Rajchman measures  motivated the definition of the  \emph{Rajchman algebra} for general (not necessarily Abelian) locally compact groups.  Since a non-Abelian group does not  admit a Fourier transform in the classical sense,  non-commutative harmonic analysis has been instead founded upon a representation theoretic point of view (see \cite{eymard}). In particular, the Rajchman algebra of a locally compact group $G$, denoted by $B_0({G})$, is defined as the collection of all coefficient functions associated with unitary representations of $G$ which vanish at infinity. If $G$ is Abelian, the Fourier-Stieltjes transform maps $M_0(G)$ onto  the Rajchman algebra of the dual group. The significance of the Rajchman algebra is due to the fact that understanding the asymptotic behavior of unitary representations is  important because of its applications in other areas of mathematics such as the theory of automorphic forms and ergodic properties of flows on homogeneous spaces ({\it e.g.}  see \cite{Howe-Moore}, \cite{Moore-flow}, and \cite{Shintani}). The Rajchman algebra of a locally compact (not necessarily Abelian) group has been studied by several researchers over the past few decades (for example see \cite{Baggett-Taylor-A=B0=>reducible, Baggett-Taylor78, Baggett-Taylor82, B0, knudby1, Knudby2, KLU}).

In this paper, we investigate Banach algebraic features, in particular  {amenability properties and spectral behavior},  of  the Rajchman algebra of a locally compact group.
Amenability of a group is a fundamental notion  that was originally introduced by  von Neumann in 1929.
In 1972, Johnson defined {\it amenable Banach algebras} as those on which no nontrivial, yet sensible, continuous derivatives can be defined (see \cite{johnson72} for the precise definition). 
This cohomological  notion turned out to be very fruitful, and has been studied extensively.  For several important classes of Banach algebras, amenability
identifies the ones with ``good behavior''.   For example,  Connes~\cite{connes78} and  Haagerup~\cite{haagerup83}  showed that for $C^*$-algebras amenability and nuclearity coincide. Johnson's well-known theorem characterizing amenable $L^1$-algebras \cite{johnsonM}, and  Forrest and Runde's characterization of amenable Fourier algebras \cite{forrest-runde-A} are other instances of this phenomenon. We refer the reader to  \cite{rundeB} for a detailed account on amenability of Banach algebras.

Over the past few decades, deeper investigation of amenability-type properties of Banach algebras associated with locally compact groups, including the Fourier algebra and Rajchman algebra, has become a vibrant trend of research in abstract harmonic analysis ({\it e.g.} see \cite{Bade-Curtis-Dales, DGH, Dales-Pandey, Forrest-wood, forrest-runde-A,  Lau-Loy, rundes2}). 
One of the important and fundamental questions here is the existence and construction of derivations for these algebras.  The derivation
problem is of great importance, as it sheds substantial light on the structure of the algebra, and then in turn on the underlying group.
As an important example, the still open question of characterizing groups whose Fourier algebras do not have non-zero {\it symmetric} derivatives has attracted a lot of attention (see  \cite{CG-AG, CG-AG2, forrest-runde-A, Forrest-Samei-Spronk-w.amen-A, johnson95, LLSS}). 

Amongst all derivations, point derivations play a particularly important role. However, examples of point derivations are rare, and except in a few basic instances we do not know how to construct them. In this paper, we investigate the spectral structure of Rajchman algebras, and illustrate aspects of the residual analytic structure of their maximal ideal space. When $G$ is Abelian, the algebra of Rajchman measures $M_0(G)$ is a commutative convolution measure algebra, {\it i.e.} it has a natural lattice structure which is compatible with its Banach algebra structure. In \cite{J.Taylor}, Taylor  showed that one can construct analytic discs in the spectrum of a convolution measure algebra around any element with non-idempotent modulus. It is now interesting to study the possibilities for elements of the spectrum whose modulus are idempotent. In  \cite{Brown-Moran-pt-der} ,  Brown and Moran constructed nontrivial continuous point derivations for $M(G)$ at the discrete augmentation character. In a subsequent paper, they used a method of Varopoulos to construct analytic discs around the discrete augmentation character (see \cite{Brown-Moran-disc}). In the present article, we use a suitable decomposition of $M_0(G)$ to develop analogues results for the spectrum and derivations of $M_0(G)$. In particular, we obtain  analytic discs around certain idempotent characters of $M_0(G)$, when $G$ is a non-discrete metrizable locally compact Abelian group. 

Our methods and constructions in this paper are based on the notion of strongly independent sets. In his rather difficult and technical paper \cite{varopoulos2}, Varopoulos proved that if $G$ is a non-discrete metrizable locally compact Abelian group, then there exists a perfect strongly independent subset $P$ of $G$ such that $M_0^+(P)\neq \{0\}$. Moreover, as explained in \cite{B0}, it is easy to see that one may assume $P$ to be compact. Such independent sets may be used to analyze the measure algebra of  a locally compact group and its subalgebras. For example, in \cite{varopoulos},  Varopoulos obtains a direct decomposition of the algebra of continuous measures $M_c(G)$, and hence the measure algebra $M(G)$, of a non-discrete locally compact Abelian group $G$, through a very careful use of geometric and measure theoretic features of strongly independent sets.  As a consequence of this decomposition theorem and its extension to $M_0(G)$, it was shown in \cite{B0} that the Rajchman algebra does not have a bounded approximate identity if $G$ is a non-compact connected SIN group. The above-mentioned decomposition theorem also sheds light on the cohomology of $M_0(G)$ as a Banach algebra. In particular, it was shown in \cite{B0} that if $G$ is non-discrete and Abelian, then $M_0(G)$  admits a nonzero continuous derivation into a symmetric bimodule. It is then natural to ask whether $M_0(G)$ admits a nonzero continuous point derivation, {\it i.e.} one into a one-dimensional symmetric bimodule. 
In this article, we answer this question affirmatively. 

This paper is organized as follows. In Section \ref{section-background}, we present the necessary notation and background. In particular, we give a very brief overview of the decomposition of $M(G)$, and respectively $M_0(G)$, developed by Varopoulos in  \cite{varopoulos}. We finish this section by a quick introduction on derivations on Banach algebras. 
In Section \ref{section:analytic-disc}, we discuss the analytic behavior of the spectrum of $M_0(G)$. Namely, we show that for every non-discrete metrizable locally compact Abelian group $G$, the maximal ideal space of $M_0(G)$ contains an analytic disc.
Inspired by the construction of analytic discs, we prove in Section \ref{section:pt-der} that $M_0(G)$ has a nonzero continuous point derivation, if $G$ is a non-discrete locally compact Abelian group. We then give the definition of the Rajchman algebra of a non-Abelian locally compact group, and extend the result for Abelian groups to the case of SIN groups. 
\section*{Acknowledgements}
The research presented in this article was part of the authors Ph.D. thesis. I would
like to thank my supervisors Brian Forrest and Nico Spronk for their encouragement
and invaluable discussions and suggestions. Many thanks to H. Garth Dales and Colin Graham for their suggestions. 

\section{A decomposition theorem for $M_0(G)$}\label{section-background}

\paragraph{Notations and key definitions}
Let $G$ be a locally compact group with Haar measure $\lambda_G$.  The group algebra of $G$, denoted by $L^1(G)$, is the Banach $*$-algebra of integrable functions on $G$  equipped with pointwise addition and convolution product. Let $M(G)$ denote the Banach $*$-algebra of
bounded regular Borel measures on $G$, equipped with the total variation norm, where  the convolution product of two measures $\mu$ and $\nu$ in $M(G)$ is defined to be
$$\int_Gf(z)d(\mu*\nu)(z)=\int_G\int_Gf(xy)d\mu(x)d\nu(y),$$
for every $f$ in $C_c(G)$, the set of compactly supported continuous functions on $G$. 
The measure algebra $M(G)$ contains $L^1(G)$ as a closed ideal. 
Let $M_c(G)$ denote the set of all continuous measures in $M(G)$, {\it {{\it {i.e.}}}} the set of complex bounded Radon measures on $G$ which annihilate every singleton. Two measures $\mu$ and  $\nu$  in $M(G)$ are called \emph{mutually singular}, denoted by $\mu\bot \nu$,
if  there exists a partition $A\cup B$ of $G$ such that $\mu$ is concentrated in $A$ and $\nu$ is concentrated in $B$.
We say $\mu$ is \emph{absolutely continuous with respect to} $\nu$, denoted by $\mu\ll \nu$, if for every measurable set $A$ for which 
$|\nu|(A)=0$, we have $|\mu|(A)=0$ as well. 

For the rest of this article, we assume that $G$ is an \emph{Abelian} locally compact group, unless otherwise is stated. Let $\widehat{G}$ denote the Pontryagin dual of $G$, {\it i.e.} the  group consisting of all continuous  homomorphisms from $G$ into ${\mathbb T}$, together with the natural group operations and the compact-open topology ({\it i.e.} uniform convergence on compact sets). It turns out that $\widehat{G}$ is a locally compact Abelian group as well.  The \emph{Fourier transform of $G$} is the norm-decreasing  $*$-homomorphism ${\cal F}:L^1(G)\rightarrow C_0(\widehat{G})$ defined as 
$${\cal F}f(\chi)=\int_Gf(x)\overline{\chi(x)}d\lambda_G(x).$$
Let $C_b(G)$ denote the Banach $*$-algebra of bounded continuous functions on $G$. The Fourier transform can be extended to the \emph{Fourier-Stieltjes transform}, which is again a norm-decreasing  $*$-homomorphism defined as ${\cal FS}:M(G)\rightarrow C_b(\widehat{G})$, 
$${\cal F}\mu(\chi)=\int_G\overline{\chi(x)}d\mu(x),$$
for every $\mu\in M(G)$. See \cite{rudinbook} for a beautiful exposition on Fourier analysis of Abelian groups. 

The \emph{algebra of Rajchman measures} of $G$,  denoted by $M_0(G)$, is the collection  of all measures in $M(G)$ whose Fourier transforms vanish at infinity. 
For a measurable subset $E$ of $G$, define $M_c(E)$ (respectively $M_0(E)$) to be the set of all measures in $M_c(G)$ (respectively $M_0(G)$) which are supported in the set $E$.
It is easy to see that $M_c(G)$ and $M_0(G)$ are closed ideals of $M(G)$, and $M_0(G)$ is contained in $M_c(G)$. It turns out that $M_0(G)$ is also ``closed'' in the sense of absolute continuity, as formalized below. 
\begin{definition}\label{def:L-space}
A closed subspace $B$ of $M(G)$ is called an \emph{$L$-space} if  for every $\mu,\nu\in  M(G)$ satisfying $\nu\in B$ and $\mu\ll\nu$, we have $\mu\in B$.
\end{definition}
We remark that $L$-spaces are sometimes referred to as  ``bands'', for example this is the terminology  used in \cite{varopoulos}. We refer the reader to \cite{bourbaki-book} for more information on $L$-spaces. 
In the following lemma, we give an equivalent definition of an $L$-space, where ``$\ll$'' is replaced by ``$\leq$''.
\begin{lemma}\label{lem:band-properties1}
Let $B$ be a closed subspace of $M(G)$. Then
$B$ is an $L$-space if and only if it satisfies the following condition:
\begin{equation}\label{cond-band}
\mbox{If } \mu,\nu\in  M(G), \nu\in B, \mbox{ and } |\mu|\leq|\nu|, \mbox{ then } \mu\in B.
\end{equation}
\end{lemma}
\begin{proof}
First assume that $B$ is an $L$-space. Note that for measures $\mu$ and $\nu$  in $M(G)$, the inequality  $|\mu|\leq |\nu|$ clearly implies $\mu\ll\nu$, just by definition. Therefore
$B$ clearly satisfies Condition~(\ref{cond-band}) as well.

Conversely, assume that $B$ is a closed subspace of $M(G)$ that satisfies Condition~(\ref{cond-band}). Suppose $\mu,\nu\in M(G)$ and 
$\nu\in B$. First note that by Condition (\ref{cond-band}), $|\nu|$ belongs to $B$ as well.
Now assume that $\mu\ll\nu$, i.e. $|\mu|\ll|\nu|$.
So, by Radon-Nikodym Theorem, there exists  a non-negative Borel integrable  function $f$ such that  $|\mu|=f|\nu|$.  For each $n
\in {\mathbb N}$, let $f_n$ be defined by
$$f_n(x)=\left\{
\begin{array}{cc}
   f(x) & \mbox{ if } f(x)\leq n \\
    n & \mbox{ otherwise }
  \end{array}
\right..
$$
 Note that $f_n|\nu|\leq n|\nu|$, which implies that $f_n|\nu|$ belongs to $B$. 
 On the other hand, by monotone convergence theorem, 
 $$\|f|\nu|-f_n|\nu|\|_{M(G)}=(f|\nu|-f_n|\nu|)(G)=\int_Gf d|\nu|-\int_Gf_n d|\nu|\rightarrow 0.$$
 Therefore $f|\nu|$, being the limit of $f_n|\nu|$'s, belongs to $B$ as well.
 \end{proof}
In the following lemma,  we list some easy properties of $L$-spaces, that we will use in future.
\begin{lemma}\label{lem:band-properties2}
Let $B$ be an $L$-space in $M(G)$, and $\mu\in M(G)$.
\begin{itemize}
\item[(a)] If $\mu\in B$ then $|\mu|\in B$ and ${\cal R}\mu, {\cal I}\mu\in B$, where ${\cal R}\mu$ and ${\cal I}\mu$ denote the real and imaginary parts of the measure $\mu$.
\item[(b)] If $|\mu|\in B$ then $\mu\in B$.
\end{itemize}
\end{lemma}
\begin{proof}
\begin{itemize}
\item[(a)] Suppose $\mu\in B$. From the definition of an $L$-space, it is clear that $|\mu|$, ${\cal R}\mu$ and ${\cal I}\mu$ belong to $B$ as well, since these measures are all absolutely continuous with respect to $\mu$. 
\item[(b)] This follows trivially from the definition of an $L$-space.
\end{itemize}
\end{proof}
It is known that $M_0(G)$ is a translation-invariant $L$-subspace of $M(G)$ (for example see \cite{Grahambook} for a proof). This important feature of $M_0(G)$ will be used in future extensively. In the following lemma, we use properties of $L$-spaces to give an easy proof of  the well-known fact that $M_0(G)$ is contained in $M_c(G)$.

\begin{lemma}\label{lem:M0inMc}
For a non-discrete  locally compact Abelian group $G$, $M_0(G)\subseteq M_c(G)$.
\end{lemma}
\begin{proof}
Suppose $M_0(G)\nsubseteq M_c(G)$ and let $\mu\in M_0(G)\setminus M_c(G)$. Note that by Lemma \ref{lem:band-properties2}, the real and imaginary components of $\mu$, denoted by ${\cal R}\mu$ and ${\cal I}\mu$, belong to $M_0(G)$ as well. Moreover, at least one of ${\cal R}\mu$ or ${\cal I}\mu$ is not a continuous measure, otherwise $\mu\in M_c(G)$. Hence without loss of generality, we can assume that $\mu$ is a real measure. Let $\mu=\mu_1+\mu_2$ be the orthogonal decomposition of $\mu$ into
$\mu_1$ in $M_c(G)$ and $0\neq\mu_2$ in $\Delta(G)$, where $\Delta(G)$ is the subalgebra of $M(G)$ consisting of all discrete measures on $G$. Clearly $\mu_2\ll\mu$, which implies that $\mu_2$ belongs to $M_0(G)$. Thus, there exists some $g$ in $G$ so that the point mass $\delta_g$
belongs to $M_0(G)$ as well. However, $|\widehat{\delta_g}(\chi)|=|\overline{\chi(g)}|=1$ for every $\chi\in \widehat{G}$, which is a contradiction with the definition of $M_0(G)$. Here, we used the well-known fact that if $G$ is non-discrete then its dual group $\widehat{G}$ is non-compact. 
\end{proof}

\paragraph{Strongly independent sets}\label{subsection:strongly-indep}
Throughout this section, we assume that $G$ is a locally compact {Abelian} group, and we present a very brief overview of the decomposition of $M(G)$, and respectively $M_0(G)$, developed by Varopoulos in \cite{varopoulos}. We begin by the definition of a strongly independent set, which is the main tool for analysis and constructions  in  \cite{varopoulos}.  We refer the reader to \cite{varopoulos} for more details and proofs on this topic.

For a subset  $P$ of a locally compact Abelian group $G$, let $k(P)$, called the \emph{torsion} of $P$, denote the smallest positive integer $k$ such that $\{kx:x\in P\}=\{0_G\}$, if such an integer exists. Otherwise, set $k(P)=\infty$. A \emph{reduced sum} on a  subset $P$ of torsion $k(P)$ is a formal expression $\sum_{i\in I}\dot{n_i}p_i$, where $I$ is a possibly empty finite index set,  $p_i$'s are distinct elements of $P$, and
$$0\neq \dot{n_i}\in \mathbb{Z}({\rm  mod }\  k(P)).$$
\begin{definition}
A subset $P\subseteq G$ is called \emph{strongly independent} if for any positive integer $N$, any family $\{p_j\}_{j=1}^N$ of distinct elements of $P$, and any family of integers $\{n_j\}_{j=1}^N$, the equality $\sum_{j=1}^N n_jp_j=0_G$ implies that $n_j$ is a multiple of $k(P)$ for each $1\leq j\leq N$, unless $k(P)=\infty$, in which case $n_j=0$ for each $1\leq j\leq N$.
\end{definition}
Let ${\rm Gp}(P)$ denote the group generated by $P$ in $G$. The computational advantage of a strongly independent set $P$ lies in that fact that every $x$ in ${\rm Gp}(P)$  can be expressed uniquely (up to permutation) as a reduced sum. Note that if $G$ is a non-discrete locally compact Abelian group then $G$ has a perfect metrizable subset $P$ which is strongly independent
\cite{varopoulos2}. Moreover, if $G$ is metrizable as well, we can assume that the above-mentioned subset $P$ satisfies the additional condition  $M_0^+(P)\neq \{0\}$,  as stated in the following theorem.
\begin{theorem}\cite{varopoulos2}\label{thm:varop2}
Let $G$ be a non-discrete metrizable locally compact Abelian group. Then there exists a perfect strongly independent subset $P$ of $G$
so that
$$M_0(P):=\Big\{\mu\in M_0(G):\supp(\mu)\subseteq P\Big\}\neq \{0\}.$$
\end{theorem}
The proof of the above theorem is rather difficult and technical. In fact, the argument in \cite{varopoulos2} relies on structural theorems and treatment of some special groups. A weaker version of Theorem 
\ref{thm:varop2} was proved by Rudin in \cite{Rudin-independent-set-M0} for the special case where $G={\mathbb T}$.
The following lemma is an easy consequence of Theorem \ref{thm:varop2} together with measure-theoretic features of $M_0(G)$, and  will be used to construct nonzero continuous point derivations on $M_0(G)$ later on.
\begin{lemma}\label{lem:compact-strongly-indep-M_0nonzero}
Let $G$ be a non-discrete metrizable locally compact Abelian group. Then there exists a compact perfect strongly independent subset $P$ of $G$
such that $M_0^+(P)\neq \{0\}$.
\end{lemma}
\begin{proof}
By Theorem \ref{thm:varop2} there exists a perfect metrizable strongly independent subset  $P'$ of $G$ which supports a nonzero Rajchman measure $\mu_0$. It is known that $M_0(G)$ is an $L$-space \cite{Grahambook}. Therefore by Lemma \ref{lem:band-properties2}, without loss of generality
we can assume that $\mu_0$ is a positive measure. Since $\mu_0(P')>0$ and $\mu_0$ is a Radon
measure,  there exists a compact subset $K$ of $P$ for which $\mu_0(K)>0$.
Note that $\mu_0|_K$ is a positive measure supported in $K$ and dominated by $\mu_0$. So by Lemma \ref{lem:band-properties1},  $\mu_0|_K$ belongs to  $M_0(K)=M_0(G)\cap M(K)$.  Moreover, ${\rm supp}(\mu_0|_K)$ is still a perfect set, because $\mu_0|_K$ is a continuous  measure according to Lemma \ref{lem:M0inMc}.
Let $P={\rm supp}(\mu_0|_K)$. Clearly $P$ is a strongly independent set, since it is a subset of the strongly independent set $P'$.
Hence $P$ is a compact perfect strongly independent
subset of $G$ with $M_0(P)\neq \{0\}$.
\end{proof}

\paragraph{A decomposition of $M(G)$ and $M_0(G)$}\label{subsection:M(G)decomp} 
Let $G$ be a non-discrete metrizable locally compact Abelian group.
Fix a strongly independent compact perfect metrizable subset $P$ of $G$ which satisfies $M_0(P)\neq \{0\}$. For every $n\in {\mathbb N}$, let $B_n$ denote the $L$-space generated by products of $n$ elements of $M_c(P)$, {\it {i.e.}}
$$B_n=\Big\{\mu\in M(G):\mu\ll \mu_1*\ldots *\mu_n \mbox{ for some } \mu_1,\ldots,\mu_n\in M_c(P)\Big\}.$$
Clearly, $B_n$ is an  $L$-subspace of $M(nP)$ for every $n\in {\mathbb N}$. 
Using geometric and measure theoretic properties of strongly independent sets, it was shown in  \cite{varopoulos} that $\delta_g*B_n\perp \delta_{g'}*B_m$ whenever $(n,g)\neq (m,g')$, $n,m\in {\mathbb N}$, and $g, g'\in G$ (see Lemma 3.1 of \cite{varopoulos} for the proof). 
The $L$-spaces $B_n$ are used as building blocks for the following decomposition of $M(G)$.

\begin{theorem}\label{thm:decomposition-M}\cite{varopoulos}
Let $G$ be a non-discrete  locally compact Abelian group, and 
$P$ be a perfect metrizable strongly independent subset of $G$.
Let $\Pi=\oplus_{g\in G, n\in {\mathbb N}}\delta_g*B_n$ and $I=\Pi^{\perp}\cap M_c(G)$. Then 
\begin{itemize}
\item $\Pi$ is a translation-invariant $L$-subspace of $M_c(G)$, which is a closed subalgebra as well.  
\item $I$ is an ideal and $L$-subspace of $M(G)$.
\item One can decompose $M_c(G)$ as $M_c(G)=\Pi\oplus I$ (direct and orthogonal decomposition).
\end{itemize}
\end{theorem}

We remark that the strongly independent set $P$ plays a crucial role in the above theorem. Indeed, to guarantee the orthogonality of blocks in $\Pi$, it is necessary to use components of the form $B_n$
rather than the whole space $M_c(nP)$.
In fact, it is not even true that ``$M_c(g_1+nP)\perp M_c(g_2+mP)$ for $(g_1,n)\not=(g_2,m)$''.
For instance, if $q$ is an element of $P$ then
$q+P\subseteq 2P$ and $M_c(q+P)\subseteq M_c(2P)$.

To study the spectrum of the algebra of Rajchman measures, we need analogue decompositions for $M_0(G)$, which was not explicitly stated in \cite{varopoulos}. So, we present the proof of the following theorem, which relies heavily on Theorem \ref{thm:varop2} and Theorem \ref{thm:decomposition-M}.
\begin{theorem}\label{thm:decomposition-M_0}
Let $G$ be  a non-discrete metrizable  locally compact Abelian group.  There exists a  perfect metrizable strongly independent subset $P$ of $G$ for which the decomposition of $M_c(G)$ as in 
Theorem \ref{thm:decomposition-M}, gives a nontrivial decomposition of $M_0(G)$, {\it i.e.}
$$M_0(G)=\Pi_0\oplus I_0,$$
where $\Pi_0=\Pi\cap M_0(G)$ is a closed subalgebra and $I_0=I\cap M_0(G)$ is an ideal of the Banach algebra $M_0(G)$. In addition, both $\Pi_0$
and $I_0$ are nontrivial $L$-subspaces of $M(G)$.
\end{theorem}
\begin{proof}
Since $G$ is a non-discrete metrizable locally compact Abelian group, Theorem \ref{thm:varop2} guarantees  a perfect metrizable strongly independent subset $P$ of $G$ such that $M_0(P)\neq\{0\}$.
Let $\mu$ be an element of $M_0(G)$. Since $M_0(G)$ is a subset of $M_c(G)$, we can  decompose $\mu$ into an orthogonal sum
$$\mu=\mu_1+\mu_2,$$
with $\mu_1\in\Pi$ and $\mu_2\in I$.
Note that $|\mu_1|\ll |\mu|$ and $|\mu_2|\ll|\mu|$. Therefore $\mu_1$ and $\mu_2$ belong to $M_0(G)$, since  $M_0(G)$ is an $L$-space containing $|\mu|$. Thus, $M_0(G)=\Pi_0\oplus I_0$.

We now need to show that the decomposition is nontrivial. First note that $\Pi_0\neq\{0\}$, since 
 $\{0\}\neq M_0(P)\subseteq \Pi_0$.
 To show $I_0\neq \{0\}$, let $\mu\in M_0^+(G)$, and pick a positive measure $\nu\in I$. (Note that such measures exist, as $M_0(G)$ and $I$ are non-trivial $L$-spaces).
 Let $E$ and $F$ be compact subspaces of $G$ such that $\mu(E)>0$ and $\nu(F)>0$.  Then, 
  \begin{eqnarray*}
  \mu *\nu(E+F)&=&\int_G\int_G\chi_{E+F}(x+y)d\mu(x)d\nu(y)\\ &\geq& \int_G\int_G\chi_{E}(x)\chi_F(y)d\mu(x)d\nu(y)=\mu(E)\nu(F)>0.
  \end{eqnarray*}
This finishes the  proof, as $I_0$ contains $\mu*\nu$.
\end{proof}
%
\paragraph{Derivatives on Banach algebras}
Let ${\cal A}$ be a Banach algebra, and $X$ be a Banach space.  The space $X$ is a Banach ${\cal A}$-bimodule if it is an ${\cal A}$-bimodule
whose module actions are  continuous, {\it i.e.} there exists a positive constant $K$ such that
$$\|a\cdot x\|\leq K\|a\|\|x\|\mbox{ and }\|x\cdot a\|\leq K\|x\|\|a\|,$$
for every $x$ in $X$ and $a$ in ${\cal A}$.
A bounded linear map $D$ from ${\cal A}$ to an ${\cal A}$-bimodule $X$ is called a \emph{derivation} if for all $a$ and $b$ in ${\cal A}$,
$$D(ab)=D(a)\cdot b+a\cdot D(b).$$
Let ${\cal A}$ be a commutative Banach algebra, and $\phi$ be a character on ${\cal A}$, {\it i.e.} an algebra homomorphism  from ${\cal A}$ into ${\mathbb C}$. Then, one can turn ${\mathbb C}$ into 
a Banach ${\cal A}$-bimodule, which will be denoted by ${\mathbb C}_\phi$, using the following natural left and right module actions:
$$a\cdot \lambda=\phi(a)\lambda=\lambda\cdot a.$$
A derivation $D$  from ${\cal A}$ to ${\mathbb C}_\phi$ is called a \emph{point derivation}. The following hereditary property will allow us to drop the condition of ``metrizability'' for our main theorem.
\begin{theorem}\label{thm:heredit}{\bf (Hereditary properties)}
Let ${\cal A}$ and ${\cal B}$ be Banach algebras. Let $\alpha$ be a surjective homomorphism from ${\cal A}$ to ${\cal B}$, and $X$ be a Banach ${\cal B}$-bimodule. 
If $D:{\cal B}\rightarrow X$ is a derivation, then
\begin{itemize}
\item[(i)] $X$ is a Banach ${\cal A}$-bimodule, when equipped with the module actions
$$a\cdot_{\cal A} x:=\alpha(a)\cdot x, \ x\cdot_{\cal A} a:=x\cdot \alpha(a), \mbox{ for every } a\in {\cal A}, x\in X.$$
\item[(ii)] The composition map $D\circ\alpha$ is a derivation of ${\cal A}$ into the ${\cal A}$-module $X$.%
\end{itemize}
\end{theorem}
\begin{corollary}\label{cor-lift}
Let $\alpha$ be a surjective homomorphism from ${\cal A}$ to ${\cal B}$, where ${\cal A}$ and ${\cal B}$ are commutative Banach algebras. . If ${\cal B}$ admits a nonzero point derivation, then ${\cal A}$ will admits such a derivation as well. 
\end{corollary}
The above theorem is very easy to prove. We refer the reader to \cite{helemski-book} for a comprehensive treatment of derivations on Banach algebras. 
\section{Analytic discs in the spectrum of $M_0(G)$}\label{section:analytic-disc}
Let $G$ be a non-discrete locally compact Abelian group.
It is known that the maximal ideal space of $L^1(G)$ can be identified with $\widehat{G}$, the character group of $G$. In analogy with this result, Taylor  described the maximal ideal  space of $M_0(G)$ as the set $\widehat{S}$ of all semicharacters of a certain compact topological semigroup $S$ (see  \cite{J.Taylor}).  A \emph{semicharacter} of a topological semigroup $S$ is a nonzero continuous function of norm at most 1  that
satisfies
$f(st)=f(s)f(t)$
for every $s,t\in S$. We use the dual pairing $\langle s,\phi\rangle$ to denote the action of $\phi\in \widehat{S}$ on $s\in S$. In its general form, Taylor's result describes the spectrum of  any \emph{convolution measure algebra}, {\it {\it {i.e.}}} a closed subalgebra of $M(G)$ which is also an $L$-space.
\begin{theorem}\cite{J.Taylor}\label{taylor-semigroup}
Let ${\cal M}$ be a commutative convolution measure algebra with maximal ideal space $\Delta$.
Then there exists a compact Abelian topological semigroup $S$ and a map
$$\iota: \widehat{S}\rightarrow\Delta$$
 such that $\iota$ is a bijection, and $\widehat{S}$ separates the points of $S$.
\end{theorem}
\begin{definition}
The semigroup $S$ introduced in  Theorem \ref{taylor-semigroup} is called the structure semigroup of ${\cal M}$.
\end{definition}
For the rest of this article, $S$ denotes the structure semigroup of  a convolution measure algebra ${\cal M}$ with maximal ideal space $\Delta$. Every element of $\widehat{S}$ admits a polar decomposition as in Lemma \ref{lem:polar}. To obtain this decomposition,  define
$$H=\left\{f\in \widehat{S}: |f(s)|=0 \mbox{ or }1 \mbox{ for all } s\in S\right\}.$$ 
For every $f\in \widehat{S}$, let $\supp(f):=\overline{\{s\in S: f(s)\neq 0\}}$. The support of $f$ is an {open} compact sub-semigroup of $S$ whose complement is an ideal (see Lemma 3.2 and its corollary in \cite{J.Taylor}). We can now state the polar decomposition in $\widehat{S}$.

\begin{lemma}[Lemma 3.3 of \cite{J.Taylor}] \label{lem:polar}
If  $f\in \widehat{S}$, then $|f|\in \widehat{S}$. Moreover, there exists a unique $h\in H$ such that $f=|f|h$ and $\supp(f)=\supp(h)$.
\end{lemma}

Let ${\cal M}$ be a convolution measure algebra with maximal ideal space $\Delta$ and structure semigroup $S$. An \emph{analytic disc} in the maximal ideal space $\Delta$ is an injection $\psi$ of the open unit
disc of ${\mathbb C}$ into $\Delta$ such that $e({\mu})\circ\psi$ is holomorphic for every $\mu$ in ${\cal M}$. Here, $e(\mu)$ denotes the function which evaluates every element of $\Delta$ at $\mu$.
It is easy to see that for a semicharacter  $\phi$, if  $|\phi|$ is not an idempotent then there exists an analytic disc around $\phi$. Indeed, let $\phi=|\phi|h_\phi$ be the polar decomposition of $\phi$. Then the map $z\mapsto |\phi|^zh_\phi$ is a vector-valued analytic map from $\{z\in {\mathbb C}: \Re{z}>0\}$ to $\widehat{S}$. We recall the following result from \cite{Brown-Moran-disc}, and present a more detailed proof.
\begin{corollary}
Let ${\cal M}$ be a convolution measure algebra with structure semigroup $S$.
Let $\phi$ be an element of $\widehat{S}$ such that $|\phi|$ is not an idempotent. Then ${\cal M}$ admits a point derivation at $\phi$.
\end{corollary}
\begin{proof}
Note that for each $\mu$ in ${\cal M}$, the map $z\mapsto \langle\mu, |\phi|^zh_\phi\rangle$  is a complex-valued  analytic function on $\{z\in {\mathbb C}: \Re{z}>0\}$. We then define
$$D:{\cal M}\rightarrow {\mathbb C},\quad D(\mu)=\frac{d}{dz}(\langle\mu, |\phi|^zh_\phi\rangle)|_{z=1}.$$
It is easy to check that $D$ is a continuous point derivation on ${\cal M}$ at the character $\phi$. Moreover, using the polynomial expansion of $z\mapsto |\phi|^zh_\phi$ around $z=1$ and the Gelfand representation of ${\cal M}$, we see that $D$ is nonzero.
\end{proof}

Taking the above discussion into account, constructing $\phi$-derivations is a non-trivial (and in most cases even challenging) task, only when $|\phi|$ is an idempotent. For the special case of $M(G)$ and the discrete augmentation character $h$, Brown and Moran have constructed nontrivial continuous point derivations at $h$ (see  \cite{Brown-Moran-pt-der}). Later on, they used a method of Varopoulos to construct analytic discs around $h$ in the maximal
ideal space of $M(G)$.
Inspired by those results, we will obtain parallel constructions for the algebra of Rajchman measures. Especially, we construct analytic discs around certain idempotent characters of $M_0(G)$.

We use the following scheme, which is inspired by the method  Brown and Moran have developed for
the case of measure algebras \cite{Brown-Moran-disc}.
Let $M_0(G)=I\oplus A$ be a decomposition of $M_0(G)$ where $I$ is an $L$-ideal and $A$ is an $L$-subalgebra. Clearly
\begin{equation}\label{idemp}
h(\mu)=\left\{\begin{array}{cc}
                  0 & \mu\in I \\
                  \mu(G) & \mu\in A
                \end{array}\right.
\end{equation}
 is a character on $M_0(G)$, since $I$ is an ideal of $M_0(G)$.
Suppose that there exist mutually orthogonal $L$-subspaces $A=B_0, B_1, B_2,\ldots$ of $M_0(G)$ such that
\begin{eqnarray}
&({\rm i}) & B_1\neq \{0\},\nonumber\\
&({\rm ii})& \mbox{If  }\mu\in B_n \mbox{ and } \nu\in B_m \mbox{ then } \mu*\nu\in B_{m+n}  \mbox{ for all positive integers }m,n,\label{eq-rakov}\\
&({\rm iii})& (\oplus_{n=0}^\infty B_n)^\perp \mbox{ is an }\mbox{$L$-ideal of } M_0(G).\nonumber
\end{eqnarray}
For $z$ in ${\mathbb D}$ and $\mu$ in $M_0(G)$, define
\begin{equation}\label{analytic-disc}
\langle\mu,\phi(z)\rangle=\left\{\begin{array}{cc}
                       \int_G z^n d\mu & \mu\in B_n\\
                       0  & \mu\in (\oplus_{n=0}^\infty B_n)^\perp
                     \end{array}
                     \right.,
                     \end{equation}
where we use the convention $0^0=1$.
One can easily verify that $\phi(z)$ is an element of the maximal ideal space of $M_0(G)$, and $\phi(0)=h$.
Hence $\phi$ is an analytic disc around $h$.

\begin{proposition}\label{prop:disc}
For every non-discrete metrizable locally compact Abelian group $G$,  the maximal ideal space of $M_0(G)$ contains an analytic disc.
\end{proposition}
\begin{proof}
Following the above argument, we only need to find a nontrivial decomposition $M_0(G)=A\oplus I$ and pairwise orthogonal $L$-subspaces
$B_0$, $B_1, \ldots$  as described earlier.
Note that in a metrizable space, every perfect strongly independent compact set $K$ is totally disconnected, and is therefore homeomorphic to a standard Cantor set. Hence we can split any such $K$ into subsets  $K_1$ and $K_2$,  so that each component is again compact, perfect, and strongly independent. Moreover, we can choose $K_1$ and $K_2$ such that both $M_0(K_1)$ and $M_0(K_2)$ are nontrivial. 

Let $K_1$, $K_2$ and $K_1\cup K_2$ be perfect metrizable strongly independent compact subsets of $G$ constructed as above, such that $M_0(K_1)$ and $M_0(K_2)$ are both nontrivial.
By Theorem \ref{thm:decomposition-M_0}, we can use the set $K_1\cup K_2$ to  decompose the algebra of Rajchman measures as
$M_0(G)=(\Pi\cap M_0(G))\oplus (I_\Pi\cap M_0(G))$, where  $\Pi=\oplus_{g\in G, n\in {\mathbb N}} \delta_g*C_n$, and
$$C_n=\Big\{\mu\in M(G):\mu\ll \mu_1*\ldots *\mu_n \mbox{ for some } \mu_1,\ldots,\mu_n\in M_c(K_1\cup K_2)\Big\}.$$
We will refine this decomposition, and obtain the structure presented in (\ref{eq-rakov}) as follows. First, we apply Theorem \ref{thm:decomposition-M_0} again, this time with the strongly independent set $K_1$, to get the decomposition $M_0(G)=A\oplus I,$ { where  } 
\begin{eqnarray}
A&=&\left(\bigoplus_{g\in G, n\in {\mathbb N}} \delta_g*C^{1}_n\right)\cap M_0(G),\label{eq:decomp-analytic1}\\
C^1_n&=&\Big\{\mu\in M_c(G):\mu\ll \mu_1*\ldots *\mu_n \mbox{ for some } \mu_1,\ldots,\mu_n\in M_c(K_1)\Big\}.\label{eq:decomp-analytic2}
\end{eqnarray}
Let $B_0=A$, and for each $n\in {\mathbb N}$, let $B_n$ be the translation-invariant $L$-space defined as
$$B_n=\left(\bigoplus_{g,g'\in G}\delta_g*C_n^2\oplus \delta_{g'}*C_n^3 \right)\cap M_0(G),$$
where 
\begin{eqnarray*}
C_n^2&=&\Big\{\mu\in M(G):\mu\ll \mu_1*\ldots *\mu_n*\nu, \ \mu_1,\ldots,\mu_n\in M_c(K_2), \nu\in A\Big\},\\
C_n^3&=&\Big\{\mu\in M(G):\mu\ll \mu_1*\ldots *\mu_n, \ \mu_1,\ldots,\mu_n\in M_c(K_2)\Big\}.
\end{eqnarray*}

We now observe that $\{B_n\}_{n\geq 0}$ is a collection of mutually orthogonal translation-invariant $L$-spaces.
Indeed, we show that for any $x,y,x',y'\in G$ and positive integers $m$ and $n$, we have
\begin{itemize}
\item[(i)] $\delta_x*C^{1}_n$ and $\delta_y*C^2_m\oplus \delta_{y'}*C^3_m$ are orthogonal.
\item[(ii)] $\delta_x*C^{2}_n\oplus \delta_{x'}*C^3_n$ and $\delta_y*C^{2}_m\oplus \delta_{y'}*C^3_m$ are orthogonal, whenever $n\neq m$.
\end{itemize}
Note that $\delta_x*C_n$ and $\delta_y*C_m$ are orthogonal  when $(x,n)\neq (y,m)$ (see Lemma 3.1 of \cite{varopoulos}).
So $\delta_x*C^1_n$ and $\delta_y*C_m^3$ are clearly orthogonal when  $n\neq m$ or $x\neq y$.
Thus to settle the question of orthogonality of $\delta_x*C^1_n$ and $\delta_y*C_m^3$, we only need to discuss the case where  $n=m$ and $x=y$. First observe that every element of $\delta_x*C_n^1$ is supported in $x+nK_1$, and every measure in 
$\delta_x*C_n^3$ is supported in $x+nK_2$. Consider an element $z$ of the intersection  $(x+nK_1)\cap (x+nK_2)$ written as
\begin{equation}\label{eq:sum1}
z=x+k_1+\ldots +k_n=x+k'_1+\ldots +k'_n,
\end{equation}
where  $k_1,\ldots,k_n\in K_1, k'_1,\ldots,k'_n\in K_2.$
After canceling  the term $x$ from both sides, and writing each sum in reduced form,  we get 
$$z-x=t_1k_{i_1}+\ldots +t_{s}k_{i_s}=t'_1k'_{j_1}+\ldots +t'_pk'_{j_p},$$
where $t_1,\ldots, t_s,t'_1,\ldots, t'_p\in {\mathbb N}$ and  ${i_1},\ldots,{i_s}, {j_1},\ldots,{j_p }\in\{1,\ldots, n\}$. 
Since $K_1$ and $K_2$ are disjoint, the above equation gives two distinct representations of an element in ${\rm Gp}(K)$, if the sum $t'_{1}k'_{j_1}+\ldots +t'_pk'_{j_p}$
contains any term from $K_2$. Since $K$ is a strongly independent set, we conclude that $k'_1$ must have been repeated (with repetition number equal to a multiple of its torsion) in Equaion (\ref{eq:sum1}). Therefore 
$$(x+nK_1)\cap (x+nK_2)\subseteq x+R_n,$$
where $R_n$ is the subset of $nK$ defined as 
$$R_n=\Big\{x_1+\ldots+x_n: x_1,\ldots,x_n\in K, \mbox{ and } x_i=x_j \mbox{ for some } 1\leq i<j\leq n\Big\}.$$
For $\mu_1,\ldots,\mu_n\in M_c(K)$, we can easily verify that $\delta_x*\mu_1*\ldots*\mu_n(R_n)=0$. Thus, $\delta_x*C^1_n$ and $\delta_x*C^3_n$ are orthogonal. A similar discussion settles the cases involving $C_n^2$, and finishes the proof of orthogonality.

We now use the collection  $\{B_n\}_{n\geq 0}$ to construct an analytic disc in the spectrum of $M_0(G)$.  It is clear that $B_n$'s satisfy the first two conditions of (\ref{eq-rakov}). Moreover, it is easy to see that 
$\bigoplus_{n\geq 0} B_n=\Pi\cap M_0(G)$, and by Theorem  \ref{thm:decomposition-M_0}, its orthogonal complement is an $L$-ideal. So every condition of 
(\ref{eq-rakov}) is satisfied, and by Lemma 1 of \cite{Brown-Moran-disc}, there exists  an analytic disc (given in Equation (\ref{analytic-disc})) in the spectrum of $M_0(G)$, whose center is the idempotent 
$h$ in (\ref{idemp}) associated with the decomposition $A\oplus I$ coming from $K_1$.
\end{proof}

\section{Point derivations on $M_0(G)$}\label{section:pt-der}
In this section, we construct nonzero point derivations on the algebra of Rajchman measures associated with non-discrete locally compact Abelian groups. Our methods and constructions are inspired by previous work of Brown and Moran \cite{Brown-Moran-pt-der},  in which they obtain nonzero continuous point derivations on the measure algebra of a non-discrete
locally compact Abelian group. Their construction is based on the decomposition of the measure algebra of a locally compact group into its discrete and continuous parts. Given that the ``discrete-continuous'' decomposition trivializes when restricted to $M_0(G)$, we need to use the very sophisticated decomposition of $M_0$ given in Theorem \ref{thm:decomposition-M_0},  together with the structure presented in (\ref{eq-rakov}),  to prove a similar result for the algebra of Rajchman measures on a non-discrete locally compact Abelian group (see Theorem \ref{thm:pt-derivation-Abelian}).

The following lemma, which  we will use in the proof of Theorem \ref{thm:pt-derivation-Abelian}, has been proved in \cite{B0} (See  Lemma 3.2 and the proof of Theorem 3.3 therein). To be self-contained, we present a summarized version of the proof here as well.
\begin{lemma}\label{lem:Mc2&P}
Let $G$ be a non-discrete locally compact Abelian group, and $P$ be a perfect metrizable strongly independent subset of $G$. Then
for each $\mu$ in $M_c(G)$, we have $\sum_{x\in G}\mu(x+P)<\infty$.
\end{lemma}
\begin{proof}
Using the definition of a strongly independent set, it is easy to see that if $x$ and $y$ are distinct elements of $G$ then $|(x+P)\cap(y+P)|\leq 2$ (see Lemma 3.2 of \cite{B0} for a detailed discussion). 
Since $\mu$ is a continuous measure on $G$, it treats the sets $x+P$ as disjoint sets, {\it {i.e.}} $\mu((x+P)\cap(y+P))=0$ for distinct elements $x$ and $y$ in $G$. Hence for any finite number of points $x_1,\ldots,x_n$ in $G$,
$$\sum_{i=1}^n|\mu(x_i+P)|\leq|\mu|(\cup_{i=1}^n(x_i+P))\leq|\mu|(G)<\infty.$$
Therefore,
$$\sum_{x\in G}|\mu(x+P)|=\sup_{I\subset G,|I|<\infty}\sum_{x\in I}|\mu(x+P)|\leq |\mu|(G)<\infty.$$
\end{proof}

\begin{theorem}\label{thm:pt-derivation-Abelian}
If $G$ is a non-discrete locally compact Abelian group, then $M_0(G)$ has a nonzero continuous point derivation.
\end{theorem}
\begin{proof}
First assume that $G$ is metrizable.
By Lemma \ref{lem:compact-strongly-indep-M_0nonzero}, there exists a compact perfect metrizable strongly independent subset $K$ of $G$
which supports a nonzero Rajchman measure $\mu_0$. Similar to our discussion in the proof of Proposition \ref{prop:disc}, we can split $K$ into $K_1$ and $K_2$ so that both sets are compact perfect metrizable strongly independent sets which support nonzero Rajchman measures. 
Using Theorem \ref{thm:decomposition-M_0}, we obtain a nontrivial decomposition $M_0(G)=A\oplus I$, where $A$ is constructed with the set $K_1$ as defined in Equation (\ref{eq:decomp-analytic1}) and Equation (\ref{eq:decomp-analytic2}). Precisely, 
$$A=\left(\bigoplus_{g\in G, n\in {\mathbb N}} \delta_g*C^{1}_n\right)\cap M_0(G), $$
where $C^1_n$ is the $L$-space generated by all possible products of $n$ elements in $M_c(K_1)$.
Clearly $\{0\}\neq M_0(K_1)\subseteq A$ and $I\neq\{0\}$. Next, define $B_1$  as in Proposition \ref{prop:disc}, {\it i.e.}
$$B_1=\left(\bigoplus_{g,g'\in G}\delta_g*M_0(K_2)\oplus \delta_{g'}*C_1^2 \right)\cap M_0(G),$$
where 
$C_1^2=\Big\{\mu\in M_c(G):\mu\ll \mu_1*\nu, \ \mu_1\in M_c(K_2), \nu\in A\Big\}$.
Observe that $B_1\subseteq I$, and as in Proposition \ref{prop:disc}, we have the direct sum decomposition 
$$M_0(G)=A\oplus B_1\oplus J,$$
which satisfies the following lattice-type property: $A*A\subseteq A$, $A*B_1\subseteq B_1$, but $A*J$, $B_1*B_1$, $B_1*J$, and $J*J$ are all subsets of $J$. 

For each $\mu$ in $M_0(G)$, let $\mu=\mu_A\oplus\mu_{B_1}\oplus \mu_J$ denote its decomposition according to the above decomposition of $M_0$.
Define  linear functionals $\chi$ and $d$ to be
$$\chi:M_0(G)\rightarrow {\mathbb C}, \ \mu\mapsto \mu_{A}(G),$$
and
$$d:M_0(G)\rightarrow {\mathbb C}, \ \mu\mapsto \mu_{B_1}(G).$$

Clearly $\chi$ is a nonzero character of $M_0(G)$, since $B_1\oplus J$ is an ideal and $A$ is a subalgebra of $M_0(G)$.
By Lemma \ref{lem:Mc2&P}, $d$ is a well-defined nonzero linear map which vanishes on $A$ and $J$.
We now show that $d$ is a point derivation of $M_0(G)$ at the character $\chi$. Fix arbitrary elements $\mu,\nu\in M_0(G)$. Using the lattice structure of the above direct sum decomposition of $M_0(G)$ and the definition of the map $d$, we have
\begin{eqnarray*}
d(\mu*\nu)&=&d(\mu_{A}*\nu_{B_1}+\mu_{B_1}*\nu_{A})\\
&=&(\mu_{A}*\nu_{B_1})(G)+(\mu_{B_1}*\nu_{A})(G)\\
&=&\mu_{A}(G)\nu_{B_1}(G)+\mu_{B_1}(G)\nu_{A}(G)\\
&=&\chi(\mu_A)d(\nu_{B_1})+d(\mu_{B_1})\chi(\nu_{A})\\
&=&\chi(\mu)d(\nu)+d(\mu)\chi(\nu),\\
\end{eqnarray*}
which finishes the proof for the metrizable case.

For the general case, let $G$ be a non-discrete locally compact Abelian group, and $H$ be a compact subgroup of $G$ such that $G/H$ is metrizable and non-discrete. Let $p$ be the quotient map from $G$ to $G/H$, and $\check{p}$ be the surjective Banach algebra homomorphism from $M_0(G)$ to $M_0(G/H)$ induced by $p$.
By the above argument, $M_0(G/H)$ has a nonzero continuous point derivation. Hence by Lemma \ref{thm:heredit}, $M_0(G)$ has a nonzero
continuous point derivation as well.
\end{proof}
\begin{remark}
Note that choosing a different perfect compact strongly independent subset $P$ instead of $K$ may result in a different decomposition for $M_0(G)$. In fact, let $K$ and $\mu_0$ be  as in Theorem \ref{thm:pt-derivation-Abelian}.
Let $P_1$ and $P_2$ be disjoint perfect subsets of $K$
such that $\mu_0$ restricts to nonzero measures on $P_1$ and $P_2$ respectively. Then for each $x$ and $y$ in $G$ and integers $m$ and $n$,  $M_c(x+mP_1)$ and $M_c(y+nP_2)$ are
orthogonal subsets of $M_c(G)$. This implies that the decomposition of $M_0(G)$ based on $P_1$ (instead of $K$) is different from the
one that is based on  $P_2$ (instead of $K$). We can now apply Theorem \ref{thm:pt-derivation-Abelian} to each decomposition and obtain distinct nonzero
continuous point derivations for $M_0(G)$.
\end{remark}
\begin{remark}
When $G$ is a discrete locally compact Abelian group, we have $L^1(G)=M(G)=M_0(G)$. However, it has been proved that the convolution algebra $L^1(G)$  never admits a nonzero point derivation (see \cite{Johnson-w.amen-L1}). 
\end{remark}\\

\paragraph{Discussion on non-Abelian groups}
The Rajchman algebra of a non-Abelian locally compact group is defined in a rather indirect way. A \emph{coefficient function} of a continuous unitary representation  $\pi:G\rightarrow {\cal U}({\cal H})$ associated with vectors  $\xi,\eta\in {\cal H}$ is defined as 
$$\xi*_\pi\eta:G\rightarrow {\mathbb C}, \quad g\mapsto \langle\pi(g)\xi,\eta\rangle.$$
Since $\pi$ is WOT-continuous, every coefficient function of $\pi$ is a continuous function on $G$.
The Rajchman algebra of a locally compact group $G$, denoted by $B_0(G)$, is defined as  the collection of all coefficient functions of $G$ which vanish at infinity. For every $u\in B_0(G)$, we define 
$\|u\|_{B_0(G)}=\inf\|\xi\|\|\eta\|$, where the infimum is taken over all possible representations of $u=\xi*_{\pi}\eta$ as a coefficient function of a continuous unitary representation of $G$. The Rajchman algebra, equipped with pointwise operations, is a commutative Banach algebra. When $G$ is Abelian, $B_0(G)$ is isometrically isomorphic with $M_0(\widehat{G})$ via the Fourier-Stieltjes transform.

One can extend Theorem \ref{thm:pt-derivation-Abelian} to non-compact connected SIN-groups using the functorial properties of their Rajchman algebras in the following sense.
A locally compact (not necessarily Abelian) group  is called a SIN-group if it has a neighborhood basis of the identity consisting of pre-compact neighborhoods which
are invariant under inner automorphisms. This is a very natural class of groups that contains all Abelian, all compact and all discrete groups.

If $H$ is a closed subgroup of a SIN-group $G$, then the restriction
map $r : B_0(G) \rightarrow B_0(H)$ is surjective (see Theorem 4.3. in \cite{B0}). This fact, together with Corollary \ref{cor-lift}, allows us to lift any point derivation on $B_0(H)$ to a point derivation on $B_0(G)$ as described in the following corollary.

\begin{corollary}
Let $G$ be a non-compact connected SIN group. Then $B_0(G)$ has a nonzero continuous point derivation.
\end{corollary}
\begin{proof}
Any non-compact connected SIN group has a copy of ${\mathbb R}^n$ as a closed subgroup for some $n\geq 1$. Recall that the restriction map $r: B_0(G)\rightarrow B_0({\mathbb R}^n)$ is a surjective homomorphism. By Theorem \ref{thm:pt-derivation-Abelian}, $B_0({\mathbb R}^n)$ has a nonzero
continuous point derivation. Applying Corollary \ref{cor-lift}, we conclude that $B_0(G)$ has a nonzero continuous point derivation as well.  
\end{proof}


\bibliographystyle{plain}

\end{document}